\documentclass[12pt,leqno,draft]{article}
\usepackage{amsfonts}
\usepackage{amsmath, amsthm, amsfonts, amssymb, color}
\usepackage{mathrsfs}
\usepackage{color}
\usepackage{stmaryrd}
\newtheorem{thm}{Theorem}[section]

\newtheorem{lem}[thm]{Lemma}
\newtheorem{prp}[thm]{Proposition}

\newtheorem{rem}[thm]{Remark}
\theoremstyle{definition}
\newtheorem{defn}{Definition}[section]
\newcommand{\scr}[1]{\mathscr #1}
\definecolor{wco}{rgb}{0.5,0.2,0.3}

\numberwithin{equation}{section} \theoremstyle{remark}

\newcommand{\ua}{\uparrow}

\title{{\bf Comparison Theorem for Distribution Dependent Neutral SFDEs}\footnote{Supported by NSFC(No., 11561027, 11661039, 71371193), NNSFC (11801406), NSF of Jiangxi(No., 20161BAB211018), Scientific Research Fund of Jiangxi Provincial Education Department(No., GJJ150444).}}

\author{
{\bf     Xing Huang$^{a)}$ and  Chenggui Yuan$^{b)}$   }\\
\footnotesize{  a)Center for Applied Mathematics, Tianjin University, Tianjin 300072, China}\\
\footnotesize{  b)Department of Mathematics, Swansea University, Swansea, SA1 8EN, UK.}\\
\footnotesize{  xinghuang@tju.edu.cn, C.Yuan@swansea.ac.uk.}}

\begin{document}
\def\R{\mathbb R}  \def\ff{\frac} \def\ss{\sqrt} \def\B{\mathbf
B}
\def\N{\mathbb N} \def\kk{\kappa} \def\m{{\bf m}}
\def\dd{\delta} \def\DD{\Delta} \def\vv{\varepsilon} \def\rr{\rho}
\def\<{\langle} \def\>{\rangle} \def\GG{\Gamma} \def\gg{\gamma}
  \def\nn{\nabla} \def\pp{\partial} \def\EE{\scr E}
\def\d{\text{\mbox{d}}} \def\bb{\beta} \def\aa{\alpha} \def\D{\scr D}
  \def\si{\sigma} \def\ess{\text{\rm{ess}}}
\def\beg{\begin} \def\beq{\begin{equation}}  \def\F{\mathscr F}
\def\Ric{\text{\rm{Ric}}} \def\Hess{\text{\rm{Hess}}}
\def\e{\text{\rm{e}}} \def\ua{\underline a} \def\OO{\Omega}  \def\oo{\omega}
 \def\tt{\tilde} \def\Ric{\text{\rm{Ric}}}
\def\cut{\text{\rm{cut}}} \def\P{\mathbb P} \def\ifn{I_n(f^{\bigotimes n})}
\def\C{\scr C}      \def\aaa{\mathbf{r}}     \def\r{r}
\def\gap{\text{\rm{gap}}} \def\prr{\pi_{{\bf m},\varrho}}  \def\r{\mathbf r}
\def\Z{\mathbb Z} \def\vrr{\varrho} \def\l{\lambda}
\def\L{\scr L}\def\Tt{\tt} \def\TT{\tt}
\def\i{{\rm in}}\def\Sect{{\rm Sect}}\def\E{\mathbb E} \def\H{\mathbb H}
\def\M{\scr M}\def\Q{\mathbb Q} \def\texto{\text{o}} \def\LL{\Lambda}
\def\Rank{{\rm Rank}} \def\B{\scr B} \def\i{{\rm i}} \def\HR{\hat{\R}^d}
\def\to{\rightarrow}\def\l{\ell}
\def\8{\infty}\def\Y{\mathbb{Y}}
\def\W{\mathbb{W}}

\def\Lip{{\rm Lip}}

\maketitle

\begin{abstract}
In this  paper, the existence and uniqueness of strong solutions to distribution dependent neutral SFDEs are proved. We give the conditions such that  the order preservation of  these equations holds.  Moreover, we show these conditions are also necessary when the coefficients are continuous. Under sufficient conditions, the result extends the one in the distribution independent case, and the necessity of these conditions  is new even in distribution independent case. 
\end{abstract}


\section{Introduction}
It is well-known that the order preservation is always an important topic in every field of mathematics. In the theory of  stochastic processes,  the order preservation is called ``comparison theorem". There are  order preservations in the distribution sense and  in the pathwise  sense,  the pathwise one  implies the distribution one. There are a lot of literature to investigate the comparison theorem. For example:  Ikeda and Watanabe  \cite{IW}, O'Brien \cite{O}, Skorohod \cite{S} and  Yamada \cite{Y} for one dimensional stochastic differential equations (SDEs) in the pathwise  sense, respectively; Chen and Wang \cite{CW}  for multidimensional diffusion processes  in the distribution  sense; Gal'cuk and Davis \cite{GD}, and Mao \cite{M} for one dimensional SDEs driven by semimartingale in the pathwise  sense, to name a few, see also \cite{W,W1}. Moreover, the comparison theorem has been extended to stochastic functional (delay) differential equations (SFDEs),   SDEs driven by jumps processes and backward SDEs, we refer reader to see  \cite{BY, HW, PY, PZ, YMY, Z}, and the references therein.

Recently, in their paper \cite{bai1}, Bai and Jiang made the contribution on the comparison theorem for  {\it neutral} SFDEs, they give  sufficient conditions such that the the comparison theorem holds for this class of stochastic equation. In present paper, we shall study
 the comparison theorem for  {\it distribution dependent neutral} SFDEs. Our results cover the ones in  \cite{bai1}. Furthermore,  we find the conditions are also necessary.

\section{Preliminaries}
Throughout the paper, we let $(\mathbb{R}^n, \<\cdot,\cdot\>, |\cdot|)$ be an $n$-dimensional Euclidean
space. Denote $\mathbb{R}^{n\times m}$ by the set of all $n\times m$ matrices endowed with Hilbert-Schmidt norm
 $\|A\|_{HS}:=\sqrt{\mbox{trace}(A^\ast A)}$ for every $A\in \mathbb{R}^{n\times m}$, in which $A^\ast$ denotes the transpose of $A$. For fixed $r_0>0$, let $\C=C([-r_0,0];\R^n)$ denote the
family of all continuous functions $h:[-r_0,0]\rightarrow\R^n$, endowed with the uniform norm
$\|h\|_\8:=\sup_{-r_0\le\theta\le0}|h(\theta)|$. Let $\mathscr{P}(\C )$ denote all probability measures on $\C$. For any continuous map $f: [-r_0,\infty)\to \R^n$  and
$t\ge 0$,  let  $f_t\in\C$ be such that $f_t(\theta)=f(\theta+t)$ for $\theta\in
[-r_0,0]$. We call $(f_t)_{t\ge 0}$ the segment of $(f(t))_{t\ge -r_0}.$ For $p\ge 2,$ let $\mathscr{P}_p(\C) $ denote all probability measures on $\C$ with finite $p-$moment, i.e. $\mu(\|\cdot\|^p_\8)=\int_\C\|\xi\|_\infty^p\mu(\d \xi)<\8.$ It is well-known that $\mathscr{P}_p(\C) $ is a polish space under the $L^p-$Wasserstein distance
$$
\mathbb{W}_p(\mu_1, \mu_2)=\inf_{\pi\in \mathbf{C}(\mu_1, \mu_2)}\bigg(\int_{\C\times \C}\|\xi-\eta\|_\8^p\pi(\d \xi, \d \eta)\bigg)^{1/p},
$$
where $\mathbf{C}(\mu_1, \mu_2)$ denotes the class of coupling of $\mu_1$ and $\mu_2.$
 Let  $(\Omega, \mathscr{F}, \{\mathscr{F}_t\}_{t\ge 0}, \mathbb{P})$ be a complete filtration
probability space,  and $\{W(t)\}_{t\ge0}$ be an $m$-dimensional standard Brownian motion defined on this probability space. For any real numbers $a, b$, we denote $a\vee b=\max\{a, b\}, a\wedge b=\min\{a, b\}, a^+=a \vee 0$ and $a^-=-(a\wedge 0).$ $a^+ (a^-)$ is called the positive (negative) part of $a.$ For a random variable $Y$ on some probability space $(E,\scr E, \P)$, we denote $\L_{Y}|\P$ the distribution of $Y$ under $\P$.
In this paper, we consider the following distribution dependent neutral   stochastic functional differential equations (NSFDEs) on $\mathbb{R}^n$:
\begin{equation}\label{neq}
\mbox{d}\big[X(t)-D(X_t)\big]=b(t, X_t, \L_{X_t})\mbox{d}t+\sigma(t, X_t, \L_{X_t})\mbox{d}W(t)
\end{equation}
and
\begin{equation}\label{neq1}
\mbox{d}\big[\bar X(t)-D(\bar X_t)\big]=\bar b(t, \bar X_t, \L_{\bar X_t})\d t+\bar\sigma(t, \bar{X}_t, \L_{\bar{X}_t})\mbox{d}W(t),
\end{equation}
where $D: \C\rightarrow\mathbb{R}^n$, which is called neutral term,  $b,\bar b: [0, \8)\times\C\times \mathscr{P}(\C)\rightarrow\mathbb{R}^n$, $\sigma,\bar{\sigma}: [0, \8)\times\C\times \mathscr{P}(\C)\rightarrow\mathbb{R}^{n\times m}$ are measurable, and $\L_{X_t}$ denotes the distribution of $X_t.$

\beg{defn}  $(1)$ For any $s\ge 0$, a continuous adapted process $(X_{s,t})_{t\ge s}$ on $\C$ is called a (strong) solution of \eqref{neq} from time $s$, if
  $$\E |D(X_{s,t})|^2+\E\|X_{s,t}\|_\infty^2 +\int_s^t \E\big\{|b(r,X_{s,r},\L_{X_{s,r}})|+\|\si(r,X_{s,r}, \L_{X_{s,r}})\|^2\big\}\d r<\infty,\ \ t\ge s,$$   and  $(X_s,(t):= X_{s,t}(0))_{t\ge s}$ satisfies $\P$-a.s.
\begin{align*}
X_s,(t) -D(X_{s,t})&= X_s,(s)-D(X_{s,s}) +\int_s^t b(r,X_{s,r}, \L_{X_{s,r}})\d r\\
 &\qquad\qquad\qquad+ \int_s^t \si(r,X_{s,r},\L_{X_{s,r}})\d W(r),\ \ t\ge s.
\end{align*}
We say that \eqref{neq} has (strong) existence and uniqueness, if for any $s\ge 0$ and $\F_s$-measurable random variable $X_{s,s}$ with $\E\|X_{s,s}\|_{\infty}^2<\infty$, the equation from
time $s$ has a unique solution  $(X_{s,t})_{t\ge s}$. When $s=0$ we    simply denote $X_{0,}=X$; i.e.\,$X_{0,}(t)=X(t), X_{0,t}=X_t, t\ge 0$.

$(2)$ A couple $(\tt X_{s,t}, \tt W(t))_{t\ge s}$ is called a weak solution to \eqref{neq} from time $s$, if $\tt W(t)$  is an $m$-dimensional standard Brownian motion on a complete filtration probability space
$ (\tt\OO, \{\tt\F_t\}_{t\ge s}, \tt\P)$, and   $\tt X_{s,t}$ solves
\begin{align*} \d (\tt X_{s,}(t)-D(\tt X_{s,t}))= b(t,\tt X_{s,t},  \L_{\tt X_{s,t}}|_{\tt\P})\d t + \si(t,\tt X_{s,t},  \L_{\tt X_{s,t}}|_{\tt\P})\d \tt W(t),\ \ t\ge s.\end{align*}

$(3)$  \eqref{neq} is said to satisfy weak uniqueness, if for any $s\ge 0$, the distribution of a weak solution $(X_{s,t})_{t\ge s}$ to \eqref{neq} from $s\ge 0$ is uniquely determined by $\scr L_{X_{s,s}}$.
\end{defn}
For future, we need the following assumptions.
\begin{enumerate}
\item[{\bf (A1)}] 
$D(0)=0$ and $D(\xi)\geq D(\eta)$ for $\xi\geq \eta$.
\item[{\bf (A2)}] There exists a constant $L>0$ such that
\begin{equation*}
\begin{split}
& |b(t, \xi, \mu)-b(t, \eta, \nu)|^2
+ |\bar b(t, \xi, \mu)-\bar b(t, \eta, \nu)|^2\\
&\le L(\|\xi-\eta\|_\infty^2+ \mathbb{W}_2(\mu, \nu)^2),\ \  t\ge 0, \xi, \eta\in \C; \mu, \nu \in \scr P_2(\C).
\end{split}
\end{equation*}
\item[{\bf (A3)}] For any $i=1, \ldots, n$,
\begin{equation*}
\begin{split}
&\sum_{j=1}^m\|\sigma_{ij}(t, \xi, \mu)-\sigma_{ij}(t, \eta, \nu)\|^2+ \|\bar \sigma_{ij}(t, \xi, \mu)-\bar \sigma_{ij}(t, \eta, \nu)\|^2\\
&\le L|\xi^i(0)-D^i(\xi)-\eta^i(0)+D^i(\eta)|^2, \ \ t\ge 0, \xi, \eta\in \C; \mu, \nu \in \scr{P}_2(\C),
\end{split}
\end{equation*}
here $L$ is in {\bf(A2)}.
\item[{\bf (A4)}]  There exists a increasing function $\beta(t) \ge 0$  such that
\begin{equation*}
\begin{split}
| b(t, 0, \delta_0)|^2+|\bar b(t, 0, \delta_0)|^2+|\sigma(t, 0, \delta_0)|^2+|\bar \sigma(t, 0, \delta_0)|^2\le \beta(t), t \ge 0,
\end{split}
\end{equation*}
where $\delta_0$ is the Dirac measure at point $0 \in \C.$
\item[{\bf (A5)}]  There exists a $\kappa\in (0, 1)$ such that
\begin{equation*}
|D(\xi)-D(\eta)|\le \kappa \max_{1\leq i\leq n}\|\xi^i-\eta^i\|_\8.
\end{equation*}
\end{enumerate}

\section{Existence and Uniqueness}
In this section, we investigate the existence and uniqueness of the solution to \eqref{neq}. To this end, we use  conditions which are weaker than the assumptions above.
\begin{enumerate}
\item[{\bf (A2')}] There exists an increasing function $\alpha:[0,\infty)\to(0,\infty)$ such that
\begin{equation*}
\begin{split}
& |b(t, \xi, \mu)-b(t, \eta, \nu)|^2
+ |\bar b(t, \xi, \mu)-\bar b(t, \eta, \nu)|^2\\
&\le \alpha(t)(\|\xi-\eta\|_\infty^2+ \mathbb{W}_2(\mu, \nu)^2),\ \  t\ge 0, \xi, \eta\in \C; \mu, \nu \in \scr P_2(\C).
\end{split}
\end{equation*}
\item[{\bf (A3')}] For $\alpha$ in {\bf(A2')},
\begin{equation*}
\begin{split}
&\|\sigma(t, \xi, \mu)-\sigma(t, \eta, \nu)\|^2+ \|\bar \sigma(t, \xi, \mu)-\bar \sigma(t, \eta, \nu)\|^2\\
&\le \alpha(t)(\|\xi-\eta\|_\infty^2+ \mathbb{W}_2(\mu, \nu)^2), \ \ t\ge 0, \xi, \eta\in \C; \mu, \nu \in \scr P_2(\C).
\end{split}
\end{equation*}
\item[{\bf (A5')}]  There exists a $\kappa\in (0, 1)$ such that
\begin{equation*}
|D(\xi)-D(\eta)|\le \kappa \|\xi-\eta\|_\8.
\end{equation*}
\end{enumerate}
\begin{thm}\label{TEU} Assume {\bf(A2')}, {\bf(A3')}, {\bf(A4)} and {\bf(A5')}, then the equation \eqref{neq} has s unique strong solution. Moreover, the weak uniqueness holds.
\end{thm}
We will prove this result by using the argument of \cite{HRW} and \cite{W16}, and we only need to consider the first equation in \eqref{neq}. For fixed  $s\ge 0$ and $\F_s$-measurable  $\C$-valued random variable $X_{s,s}$ with $\E\|X_{s,s}\|_{\infty}^2<\infty$, we construct the first equation in \eqref{neq} by iterating in distribution as follows. Firstly, let
$$X^{(0)}_{s,t}=X_{s,s}, \ \ \mu_{s,t}^{(0)}=\L_{X^{(0)}_{s,t}},\ \ t\ge s.$$
For any $n\ge 1,$ let $(X_{s,t}^{(n)})_{t\ge s}$ solve the classical neutral SFDE
\beq\label{EN}
\d (X^{(n)}_{s,}(t)-D(X^{(n)}_{s,t}))= b(t,X^{(n)}_{s,t}, \mu_{s,t}^{(n-1)}) \d t + \si(t,X^{(n)}_{s,t},\mu_{s,t}^{(n-1)})\,\d W(t),\ \  t\ge s,
\end{equation}
with $X_{s,s}^{(n)}=X_{s,s}$, where $\mu_{s,t}^{(n-1)}:=\L_{X_{s,t}^{(n-1)}}$  and $X^{(n)}_{s,t}(\theta):= X^{(n)}_{s,}(t+\theta)$ for $\theta\in [-r_0, 0]$.

\beg{lem} \label{L2.1} Assume {\bf(A2')}, {\bf(A3')}, {\bf(A4)} and {\bf(A5')}. Then, for every $n\ge 1$, the neutral SFDE $\eqref{EN}$ has a unique strong solution $X^{(n)}_{s,t}$ with
\beq\label{*2} \E\sup_{t\in [s-r_0,T]} |X^{(n)}_{s,}(t)|^2<\infty,\ \ T>s, n\ge 1.\end{equation} Moreover, for any $T>0$, there exists $t_0>0$ such that for all $s\in [0,T]$ and $X_{s,s}\in L^2(\OO\to\C;\scr F_s)$,
\beq\label{*3}
\E \sup_{ t\in [s, s+t_0]} |X^{(n+1)}_{s,}(t)-X^{(n)}_{s,}(t)|^2\le 4 \e^{-n}  \E\sup_{t\in [s,s+t_0]} |X^{(1)}_{s,}(t)|^2,\ \  n\ge 1.
\end{equation}
\end{lem}

\beg{proof} The proof is similar to that of \cite[Lemma 2.1]{W16} and \cite[Lemma 3.2]{HRW}.
Without loss of generality, we may assume that  $s=0$ and simply denote  $X_{0,}(t)=X(t), X_{0,t}=X_t, t\ge 0$.

(1) We first prove that  the SDE \eqref{EN} has a unique strong solution and \eqref{*2} holds.

For   $n=1$, let
$$ \check{b}(t,\xi)= b(t,\xi, \mu_{t}^{(0)}),\ \ \check{\si}(t,\xi)= \si (t,\xi, \mu_{t}^{(0)}),\ \ t\ge 0, \xi\in\C.$$ Then \eqref{EN} reduces to
\beq\label{EN*} \d (X^{(1)}(t)-D(X_t^{(1)}))=  \check{b}(t, X_{t}^{(1)})\d t + \check{\si}(t, X_{t}^{(1)})\d W(t),\ \ X_{0}^{(1)}=X_{0}, t\ge 0.\end{equation}
By {\bf(A2')}, {\bf(A3')}, {\bf(A4)} and {\bf(A5')}, the coefficients $ \check{b}$ and $\check{\si}$ satisfy the standard monotonicity condition which implies strong existence, uniqueness and non-explosion for
neutral SFDE \eqref{EN*}, see e.g. \cite[Theorem 2.1]{bai1}.
By  {\bf(A2')}, {\bf(A3')}, {\bf(A4)} and {\bf(A5')},     there exists an increasing function $H:\R_+\to\R_+$ such that
\beg{align*}& |b(t,\xi,\mu_t^{(0)})|^2  + \|\si(t,\xi,\mu_t^{(0)})\|_{HS}^2\\
&\le 2|b(t,\xi,\mu_t^{(0)})-b(t, 0,\mu_t^{(0)})|^2
  +   2|b(t,0,\mu_t^{(0)})|^2\\
  &+   2\|\si(t,\xi,\mu_t^{(0)})-\si(t,0,\mu_t^{(0)})\|^2_{HS}+2\|\si(t,0,\mu_t^{(0)})\|^2_{HS}\\
&\le H(t) \big\{1+ \|\xi\|_{\infty}^2+\mu_t^{(0)}(\|\cdot\|_{\infty}^2)\big\},\ \ \ t\ge 0, \xi\in \C.\end{align*}
For any $N\in [1,\infty)$ and $\tau_N:= \inf\{t\ge 0: |X^{(1)}(t)|\ge N\}$,
\beg{align*} & |X^{(1)}(t\wedge\tau_N)-D(X_{t\wedge\tau_N}^{(1)})|^2\\
  &\leq 3|X^{(1)}(0)-D(X_0^{(1)})|^2+3\left|\int^{t\wedge\tau_N}_0\si(s,X_s^{(1)},\mu_s^{(0)})\d W(s)\right|^2\\
&\quad + 3\left|\int_0^{t\wedge\tau_N}b(s, X_s^{(1)},  \mu_s^{(0)})\d s\right|^2.\end{align*}
Applying inequality $(x+y)^2\leq \frac{x^2}{p}+\frac{y^2}{1-p}$ for $p\in(0,1)$ and $x, y\geq 0$, we have
\beg{align*}
&|X^{(1)}(t\wedge\tau_N)|^2\leq \frac{|X^{(1)}(t\wedge\tau_N)-D(X_{t\wedge\tau_N}^{(1)})+D(0)|^2}{1-\kappa}+\frac{|D(X_{t\wedge\tau_N}^{(1)})-D(0)|^2}{\kappa}\\
&\le \kappa\|X_{t\wedge\tau_N}^{(1)}\|^2_{\infty}+c|D(0)|^2+c\|X^{(1)}_0\|_{\infty}^2+c\left|\int^{t\wedge\tau_N}_0\si(s,X_s^{(1)},\mu_s^{(0)})\d W(s)\right|^2\\
&\quad + c\left|\int_0^{t\wedge\tau_N}b(s, X_s^{(1)},  \mu_s^{(0)})\d s\right|^2, \ \ t\leq\tau_N\ \ \end{align*}
for some constant $c>0$.
Noting $\kappa\in(0,1)$, combining this with {\bf(A4)} and applying the BDG inequality we have
\beg{align*} &\E \sup_{s\in [-r_0,t\land \tau_N]} |X^{(1)}(s)|^2
 \le c\mathbb{E}\|X^{(1)}_0\|_{\infty}^2+c|D(0)|^2\\
 &\qquad\qquad\qquad+H(t) \E\int_0^{t\land \tau_N} \big(1+\|X_s^{(1)}\|_{\infty}^2 + \mu_s^{(0)}(\|\cdot\|_{\infty}^2)\big)\d s, \ \ t\geq 0.\end{align*}
This implies
\begin{equation*}\begin{split}
&\E \sup_{s\in [-r_0,t\land \tau_N]} |X^{(1)}(s)|^2 \le c\mathbb{E}\|X^{(1)}_0\|_{\infty}^2+c|D(0)|^2\\
&\qquad\qquad\qquad+ H(t)\int_0^{t} \big\{1+ \E \sup_{r\in [-r_0,s\land \tau_N]}|X^{(1)}(r)|^2 + \mu_s^{(0)}(\|\cdot\|_{\infty}^2)\big\}\d s,\ \ t\ge 0. \end{split}\end{equation*}
By first applying Gronwall's Lemma then letting $N\to\infty$, we arrive at
$$\E \sup_{s\in [-r_0,t]} |X^{(1)}(s)|^2<\infty,\ \ t\ge 0.$$
Therefore,   \eqref{*2}  holds  for $n=1$.

Now, assuming that the assertion holds for $n=k$ for some $k\ge 1$, we are going to show   it for $n=k+1$. Since the proof is similar to  repeat  the  argument above with
 $(X_\cdot^{(k+1)}, \mu_\cdot^{(k)}, X_\cdot^{(k)})$ replacing $(X_\cdot^{(1)}, \mu_\cdot^{(0)},X_\cdot^{(0)})$,  we omit it here.

(2) To prove \eqref{*3}, let
\beg{align*} &\xi^{(n)}(t) = X^{(n+1)}(t)- X^{(n)}(t),\\
&\LL_t^{(n)}= \si(t,X_t^{(n+1)},\mu_t^{(n)})- \si(t,X_t^{(n)},\mu_t^{(n-1)}),\\
&B_t^{(n)}= b(t,X_t^{(n+1)}, \mu_t^{(n)} ) - b(t,X_t^{(n)}, \mu_t^{(n-1)} ).\end{align*}
 By {\bf(A2')} and It\^o's formula,  there exists an increasing function $K_1:\R_+\to\R_+$ such that
\begin{align*}
|\xi^{(n)}(t)-(D(X^{(n+1)}_t)-D(X^{(n)}_t))|^2 & \le 2 \left|\int_0^t\LL_s^{(n)} \d W(s)\right|^2\\
&+  K_1(t)\int_0^t \big\{\|\xi_s^{(n)}\|_{\infty}^2 + \W_2(\mu_s^{(n)}, \mu_s^{(n-1)})^2\big\}\d s.
\end{align*}
Again using inequality $(x+y)^2\leq \frac{x^2}{\kappa}+\frac{y^2}{1-\kappa}$, we have
\begin{align*}
|\xi^{(n)}(t)|^2&\leq \kappa\|\xi_t^{(n)}\|_{\infty}^2+ \frac{2}{1-\kappa} \int_0^t\LL_s^{(n)} \d W(s)\\
&+  \frac{K_1(t)}{1-\kappa}\int_0^t \big\{\|\xi_s^{(n)}\|_{\infty}^2 + \W_2(\mu_s^{(n)}, \mu_s^{(n-1)})^2\big\}\d s.
\end{align*}
By the BDG inequality and noting $\kappa\in(0,1)$, we obtain
 \beg{align*} \E \sup_{s\in [0,t]}  |\xi^{(n)}(s)|^2&\le K_2(t) \int_0^t \Big\{\E   \|\xi_s^{(n)}\|_{\infty}^2 + \W_2(\mu_s^{(n)}, \mu_s^{(n-1)})^2\Big\} \d s\\
  & \leq K_2(t) \int_0^t \Big\{\E   \sup_{r\in[0,s]}|\xi^{(n)}(r)|^2 + \W_2(\mu_s^{(n)}, \mu_s^{(n-1)})^2\Big\} \d s,\ \ t\ge 0  \end{align*}
 for some increasing function
$K_2: \R_+\to \R_+.$

By Gronwall's Lemma, and since $\W_2(\mu_s^{(n)}, \mu_s^{(n-1)})^2\le \E\|\xi_s^{(n-1)}\|_\infty^2$, we obtain
\beg{align*} & \E \sup_{s\in [0,t]}  |\xi^{(n)}(s)|^2 \le tK_2(t) \e^{tK_2(t)} \sup_{s\in [0,t]} \W_2(\mu_s^{(n)}, \mu_s^{(n-1)})^2\\
&\le tK_2(t) \e^{tK_2(t)}
  \E  \sup_{s\in [0,t]}|\xi^{(n-1)}(s)|^2,\ \  \ t\ge 0.\end{align*}
 Taking $t_0>0$ such that $  t_0K_2(T) \e^{t_0K_2(T)}\le \e^{-1}$, we arrive  at
 $$\E \sup_{s\in [0,t_0]} |\xi^{(n)}(s)|^2\le \e^{-1}\E \sup_{s\in [0,t_0]} |\xi^{(n-1)}(s)|^2,\ \ n\ge 1.$$ Since
 $$\E \sup_{s\in [0,t_0]} |\xi^{(0)}(s)|^2\le 2  \E  \Big\{|X(0)|^2+ \sup_{s\in [0,t_0]} |X^{(1)}(s)|^2\Big\}\le 4 \E \sup_{s\in [0,t_0]}  |X^{(1)}(s)|^2,$$ we obtain \eqref{*3}.
\end{proof}
\begin{proof}[Proof of Theorem \ref{TEU}] ({\bf Existence}) For simplicity, we only consider  $s=0$ and denote $X_{0,}=X$; i.e.\,$X_{0,}(t)=X(t), X_{0,t}=X_t, t\ge 0$.

 Let $(X_t)_{t\in [0,t_0]}$ be the unique limit of $(X_t^{(n)})_{t\in [0,t_0]}$ in  Lemma \ref{L2.1}, then  $(X_t)_{t\in [0,t_0]}$ is an adapted continuous process  and  satisfies
\beq\label{A01}\lim_{n\to\infty} \sup_{t\in [0,t_0]} \W_2(\mu_t^{(n)},\mu_t)^2\le \lim_{n\to\infty} \E \sup_{t\in [0,t_0]}  |X^{(n)}(t)- X(t)|^2=0,
\end{equation}
where $\mu_t$ is the distribution of $X_t$. Rewriting \eqref{EN}, we have
$$X^{(n)}(t)-D(X^{(n)}_t)=X(0)-D(X_0)+ \int_0^t  b(s,X^{(n)}_s,\mu_s^{(n-1)}) \d s +\int_0^t\si(s,X^{(n)}_s,\mu_s^{(n-1)})\d W(s). $$
Then   \eqref{A01}, {\bf(A2')}, {\bf(A3')}, {\bf(A5')} and the dominated convergence theorem imply that $\P$-a.s.
$$X(t)-D(X_t)= X(0)-D(X_0)+\int_0^t b(s,X_s, \mu_s)\d s +\int_0^t \si(s,X_s, \mu_s)\d W(s),\ \ t\in [0,t_0].$$
Therefore, $(X_t)_{t\in [0,t_0]}$ solves \eqref{neq} up to time $t_0$. Moreover,    $ \E \sup_{s\in [0,t_0]} |X(s)|^2<\infty$  follows by  \eqref{A01}. The same holds for
$(X_{s,t})_{t\in [s,(s+t_0)\land T]}$ and $s\in [0,T]$. So, by solving the equation piecewise in time, and using the arbitrariness of $T>0$,
  we conclude that  \eqref{neq} has a strong solution
  $(X_t)_{t\ge 0}$ with
  \beq\label{*X} \E \sup_{s\in [0,t]}|X(s)|^2<\infty,\ \ \ t\geq 0.\end{equation}
{\bf Uniqueness} Let $X$ and $Y$ be two solutions to \eqref{neq}, i.e.
$$\mbox{d}\big[X(t)-D(X_t)\big]=b(t, X_t, \L_{X_t})\mbox{d}t+\sigma(t, X_t, \L_{X_t})\mbox{d}W(t),$$
and
$$\mbox{d}\big[Y(t)-D(Y_t)\big]=b(t, Y_t, \L_{Y_t})\mbox{d}t+\sigma(t, Y_t, \L_{Y_t})\mbox{d}W(t),$$
By {\bf(A2')}, we have
\begin{equation*}\begin{split}
|X(t)-Y(t)-(D(X_t)-D(Y_t))|^2\le& 2 \left|\int_0^t \{\si(s,X_s,\L_{X_s})-\si(s,Y_s,\L_{Y_s})\}\d W(s)\right|^2\\
+& \bb_1(t) \int_0^t\big\{\|X_s-Y_s\|_{\infty}^2+\W_2(\L_{X_s},\L_{Y_s})^2\big\}\d s
\end{split}\end{equation*}
for an increasing function $\beta_1:[0,\infty)\to [0,\infty)$.
Applying inequality $(x+y)^2\leq \frac{x^2}{\kappa}+\frac{y^2}{1-\kappa}$,  we have
\begin{equation*}\begin{split}
|X(t)-Y(t)|^2\le& \kappa\|X_t)-Y_t\|^2_{\infty}+\frac{2}{1-\kappa} \left|\int_0^t \{\si(s,X_s,\L_{X_s})-\si(s,Y_s,\L_{Y_s})\}\d W(s)\right|^2\\
+& \frac{\bb_1(t)}{1-\kappa} \int_0^t\big\{\|X_s-Y_s\|_{\infty}^2+\W_2(\L_{X_s},\L_{Y_s})^2\big\}\d s.
\end{split}\end{equation*}
Noting that $\W_2(\L_{X_t},\L_{Y_t})^2\le \E\|X_t-Y_t\|_{\infty}^2$, {\bf(A3')} and the BDG inequality imply that
$\gg_t:= \sup_{s\in [-r_0,t]} |X(s)-Y(s)|^2$ satisfies
$$\E\gg_t\le \beta_2(t) \int_0^t \E\gg_s \d s,\ \ t\ge 0$$
for an increasing function $\beta_2:[0,\infty)\to [0,\infty)$.
So, applying Gronwall's inequality implies
$$\E\gg_t=0, \ \ t\geq 0.$$

 ({\bf Weak uniqueness}) Since the proof is similar to that of \cite[Theorem 2.1]{W16}, we omit it here.
\end{proof}

\section{Comparison Theorem}
In order to obtain the comparison theorem for distribution dependent NSFDEs, we introduce the partial order on $\C$.  If
$x=(x_1, \cdots, x_n), y=(y_1, \cdots, y_n)\in \R^n$, we call $x\le y$ if and only if $x_i\le y_i, i=1, \ldots, n;$ $x< y$ if and only if $x\le y$
and $x\neq y; $  $x\ll y$ if and only if $x_i< y_i, i=1, \ldots, n.$ For $\xi=(\xi_1, \cdots, \xi_n), \eta=(\eta_1, \cdots, \eta_n)\in \C, $ we call $\xi\le \eta$ if and only if $\xi(\theta)\le \eta(\theta), \theta\in [-r_0, 0]; $ $\xi< \eta$ if and only if $\xi\le \eta$ and $\xi\neq \eta;$ $\xi\ll \eta$ if and only if $\xi(\theta)< \eta(\theta), \theta\in [-r_0, 0]; $ for any $\xi, \eta\in \C,$ $\xi\wedge \eta$ is defined by $(\xi\wedge\eta)_i= \xi_i\wedge\eta_i, i=1, \ldots, n.$ We also define the following partial order associated with the neutral term $D(\cdot),$ that is: $\xi\le_{D}\eta$ if and only if $\xi\le \eta $ and $\xi(0)-D(\xi)\le \eta(0)-D(\eta); $ $\xi<_{D} \eta$ if and only if $\xi\le_{D}\eta$ and
$\xi\neq \eta.$ A function $h$ on $\C$ is called increasing if $h(\xi)\le h(\eta)$ for $\xi\le \eta.$   Let $\mu_1, \mu_2 \in \mathscr{P}(\C ),$ we call
$\mu_1\le \mu_2$ if and only if $\mu_1(h)\le \mu_2(h)$ holds for all increasing function $h\in C_b(\C)$ which denotes all bounded continuous functions on $\C.$

Denote by $(X(s, \xi; t), \bar X(s, \bar \xi; t))_{t\geq s}$ the solutions to \eqref{neq}-\eqref{neq1} with $(X_s(s, \xi), \bar X_s(s, \bar \xi))=(\xi,\bar \xi)$. Let $(X_t(s, \xi), \bar X_t(s, \bar \xi))_{t\geq s}$ be the segment process.
\begin{defn}
The distribution dependent NSFDE \eqref{neq}-\eqref{neq1} is called D-order-preserving, if for any $s\ge 0$ and $\xi, \bar \xi\in L^2(\Omega \rightarrow \C, \mathscr{F}_s, \mathbb{P})$ with $\P(\xi\le_D \bar \xi)=1,$ one has
$$
\P(X_t(s, \xi)\le_D \bar X_t(s, \bar \xi),\ \ t\geq s)=1.
$$
\end{defn}
\begin{defn} A function $f:\C\to\mathbb{R}^1$ is called  $D$-increasing, if for any $\xi\leq_D\eta$, it holds $f(\xi)\leq f(\eta)$. If two probability  measures $\mu,\nu$ on $\C$ satisfying $\mu(f)\leq \nu(f)$ for any $D$-increasing function $f$, then we denote $\mu\leq _D\nu$.
\end{defn}
\begin{rem}\label{pc} In fact, if $\mu\leq_D\nu$, by \cite[Theorem 5]{KKO}, there exists $\pi\in \mathbf{C}(\mu,\nu)$ with
$$\pi(\{(\xi_1,\xi_2), \xi_1\leq _D\xi_2\})=1.$$
\end{rem}

\subsection{Sufficient Conditions for Comparison Theorem}

In this subsection, we will extend the result in \cite{bai1}  and provide sufficient conditions such that the comparison theorem holds. Due to the difficulty caused by the  distribution dependence, the generalization is not trivial.

\begin{thm}\label{main}
Let ({\bf{A1}})-{\bf(A5)} hold and $b, \bar b$ and $\si, \bar \si$ are continuous on $[0,\infty)\times \C\times \scr P_2(\C)$. Assume $\xi\leq _D\bar{\xi}$ and the following conditions hold.
\begin{enumerate}
\item[{\rm{(i)}}] The drift terms $b=(b_1,  \ldots, b_n)$ and $\bar b=(\bar b_1,  \ldots, \bar b_n)$  are continuous in $t$ and  $b_i(t, \eta, \mu)\le \bar b_i(t, \bar \eta, \bar \mu)$ for any $1\le i \le n$  provided $\mu, \bar \mu\in \mathscr{P}_2(\C)$ with $\mu\le_D\bar \mu$,  $\eta, \bar \eta\in \C$ with $\eta\le_D \bar \eta$ and $\eta^i(0)-D^i(\eta)=\bar{\eta}^i(0)-D^i(\bar{\eta})$.

\item[{\rm{(ii)}}] The diffusion terms $\si=(\si_{ij})$ and $\bar \si=(\bar \si_{ij})$ are continuous in $t$ and
$\si=\bar \si.$ Moreover, $\si_{ij}(t,\eta,\mu)$ only depends on $t$ and $\eta^i(0)-D^i(\eta)$.
\end{enumerate}
 Then
 $\P(X_t(s,\xi)\le_D \bar X_t(s,\bar\xi), t \ge s)=1.$ Thus, $\P(X_t(s,\xi)\le\bar X_t(s,\bar\xi), t \ge s)=1.$
\end{thm}
In the following, for simplicity, let $s=0$, $X(t)=X(s,\xi;t), \bar X(t)=\bar X(s,\bar \xi;t), X_D(t)=X(t)-D(X_t), \bar X_D(t)=\bar X(t)-D(\bar X_t).$

Define the following stopping times:
\begin{align*}
&\rho_i=\inf\{t >0: X^i(t)> \bar X^i(t)\}, \ \ i=1, 2, \ldots, n,\\
&\Upsilon_i=\inf\{t >0: X^i_{D}(t)> \bar X^i_{D}(t)\}, \ \ i=1, 2, \ldots, n.
\end{align*}
Let $\rho=\min\{\rho_1, \ldots, \rho_n\}$ and $\Upsilon=\min\{\Upsilon_1, \ldots, \Upsilon_n\}$.
We firstly give a modified proof of \cite[Proposition 3.1]{bai1} which extends the result there to the case that $D$ is nonlinear.
\begin{prp}\label{com}
Assume {\bf(A1)} and {\bf(A5)} hold, then we have
\begin{equation}\label{c1}
\Upsilon \le \rho \mbox{ on  } \Omega.
\end{equation}
\end{prp}
\begin{proof}
Set
\begin{align}\label{theta}\rho_i^l:=\inf\{t >0: X^i(t)> \bar X^i(t)+\frac{1}{l}\}, \ \ i=1, 2, \ldots, n; l\geq 1,
\end{align}
and
$$\rho^l:=\min\{\rho_1^l, \ldots, \rho_n^l\}.$$
Then it is easy to see that for $i=1,\cdots,n$ and $0\leq t\leq \rho_i^l$,
\begin{align}\label{X_i}
X^i(\rho_i^l)= \bar X^i(\rho_i^l)+\frac{1}{l},\ \ X^i(t)\leq \bar X^i(t)+\frac{1}{l},
\end{align}
and $\rho=\inf\{\rho^l: l\geq 1\}$. Moreover, by the definition of $\rho_i$ and $\Upsilon_i$, one has
\begin{align}\label{rhoi}
X^i(\rho_i)= \bar X^i(\rho_i),\ \ X^i(t)\leq \bar X^i(t), \ \ i=1,\cdots,n; 0\leq t\leq \rho_i,
\end{align}
and
\begin{align}\label{lami}
X^i_{D}(\Upsilon_i)= \bar X^i_{D}(\Upsilon_i),\ \ X^i_{D}(t)\leq \bar X^i_{D}(t), \ \ i=1,\cdots,n; 0\leq t\leq \Upsilon_i.
\end{align}
We only need to prove $\Upsilon\leq \rho^l$ for any $l\geq 1$ and $\omega\in\Omega$. To this end, we assume that there exists a $l\geq 1$ and $\omega_0\in\Omega$ such that $\rho^l(\omega_0)<\Upsilon(\omega_0)$. Then there exists a $1\leq n_0=n_0(\omega_0)\leq n$ such that $\rho_{n_0}^l(\omega_0)=\rho^l(\omega_0)$. Then by \eqref{lami}, we have $ [X^{n_0}_{D}(\rho_{n_0}^l)](\omega_0)\leq [\bar X^{n_0}_{D}(\rho_{n_0}^l)](\omega_0)$. This together with \eqref{X_i} implies that $\frac{1}{l}+D^{n_0}(\bar X_{\rho_{n_0}^l}(\omega_0))-D^{n_0}(X_{\rho_{n_0}^l}(\omega_0))\leq 0$. This combining with \eqref{X_i} and the monotonicity of $D$ yields
$$\frac{1}{l}+D^{n_0}(\bar X_{\rho_{n_0}^l}(\omega_0))-D^{n_0}(\bar X_{\rho_{n_0}^l}(\omega_0)+\frac{1}{l})\leq 0.$$
By {\bf (A5)}, we obtain $\frac{1}{l}-\frac{\kappa}{l}\leq 0$. Since $\kappa \in (0, 1)$, this is a contradiction. Thus, we finish the proof.
\end{proof}
\begin{rem}\label{cc} With Proposition \ref{com} in hand, repeating the proof of \cite[Theorem 3.1]{bai1}, we obtain the following result:
If $b,\bar b$ and $\sigma,\bar{\sigma}$ do not depend on the distribution, under {\bf(A1)}-{\bf(A5)}, Theorem \ref{main} holds by replacing the condition $b_i(t, \eta)\le \bar b_i(t, \bar \eta)$ in (i) with $b_i(t, \eta)<\bar b_i(t, \bar \eta)$.
\end{rem}
Now we intend to prove the distribution dependent case.
\begin{proof}[Proof of Theorem \ref{main}]
We first prove the result in Theorem \ref{main} holds by replacing the condition $b_i(t, \eta, \mu)\le \bar b_i(t, \bar \eta, \bar \mu)$ in (i) with $b_i(t, \eta, \mu)<\bar b_i(t, \bar \eta, \bar \mu)$.
For any $n\ge 0,$ let $(X_{s,t}^{(n)})_{t\ge s}$ solve \eqref{EN} with $X_{s,s}^{(n)}=\xi$ and $ X_{s,t}^{(0)}=\xi, t\geq s$. Similarly,  let $$\bar{X}^{(0)}_{s,t}=\bar{\xi}, \ \ t\geq s,$$
and $(\bar{X}_{s,t}^{(n)})_{t\ge s}$ solve \eqref{EN} with $\bar{b}$ and $\bar{\sigma}$ in place of $b$ and $\sigma$ and $\bar{X}_{s,s}^{(n)}=\bar{\xi}$. Denote $\bar{\mu}_{s,t}^{(n-1)}:=\L_{\bar{X}_{s,t}^{(n-1)}}$. We should remark that $\{\bar X_{s,t}^{(0)}\}_{t\geq s}$ and $\{X_{s,t}^{(0)}\}_{t\geq s}$ are continuous $\C$-valued process.
Without loss of generality, we assume $s=0$ and omit the subscript $s$.
$$b^n(t,\eta)=b(t,\eta,\mu_t^{(n-1)}), \ \ \sigma^n(t,\eta)=\sigma(t,\eta,\mu_t^{(n-1)}),$$
and
$$\bar{b}^n(t,\eta)=\bar{b}(t,\eta,\bar{\mu}_t^{(n-1)}), \ \ \bar{\sigma}^n(t,\eta)=\bar{\sigma}(t,\eta,\bar{\mu}_t^{(n-1)}).$$

For $n=1$, since $\mu_t^{(0)}\leq_D \bar{\mu}_t^{(0)}$, by (i) and (ii) in Theorem \ref{main}, we have
\begin{enumerate}
\item[{\rm{(1)}}] $b^1$ and $\bar{b}^1$ are continuous in $t$ and  $b^1_i(t, \eta)< \bar b^1_i(t, \bar \eta)$ for any $1\le i \le n$  provided $\eta, \bar \eta\in \C$ with $\eta\le_D \bar \eta$ and $\eta^i(0)-D^i(\eta)=\bar{\eta}^i(0)-D^i(\bar{\eta})$.

\item[{\rm{(2)}}] The diffusion terms $\si^1=(\si^1_{ij})$ and $\bar \si^1=(\bar \si^1_{ij})$ are continuous in $t$ and
$\si^1=\bar \si^1.$ Moreover, $\si^1_{ij}(t,\eta)$ only depends on $t$ and $\eta^i(0)-D^i(\eta)$.
\end{enumerate}
Then by Remark \ref{cc}, it holds $\mathbb{P}$-a.s.
$$X^{(1)}_t\leq_D \bar{X}^{(1)}_t,\ \ t\geq 0.$$
Next, assume
$\mathbb{P}$-a.s.
$$X^{(n-1)}_t\leq_D \bar{X}^{(n-1)}_t,\ \ t\geq 0.$$
Repeating the proof for $b^n$, $\sigma^n$, $\bar{b}^n$, $\bar{\sigma}^n$, $X^{(n-1)}$ in place of $b^1$, $\sigma^1$, $\bar{b}^1$, $\bar{\sigma}^1$, $X^{(0)}$, we can prove
$\mathbb{P}$-a.s.
$$X^{(n)}_t\leq_D \bar{X}^{(n)}_t,\ \ t\geq 0.$$
By \eqref{A01}, we conclude
$\mathbb{P}$-a.s.
$$X_t\leq_D \bar{X}_t,\ \ t\geq 0,$$
and
$$X_t\leq\bar{X}_t,\ \ t\geq 0.$$
Then the required assertion follows.

In general, if the Assumption (i) in Theorem \ref{main} holds, then let $(\bar{b}_\varepsilon, \bar X^\varepsilon)$ be in Lemma \ref{app} below.  By the above conclusion, we have $\mathbb{P}$-a.s.
$$X_t\leq_D \bar{X}^\varepsilon_t, \ \ t\geq 0.$$ Letting $\varepsilon$ goes to $0$, it follows from Lemma \ref{app} below and the continuity of $D$ that $\mathbb{P}$-a.s.
$$X_t\leq_D \bar{X}_t, \ \ t\geq 0,$$
and  $$X_t\leq \bar{X}_t, \ \ t\geq 0.$$
Thus, we complete the proof.
\end{proof}
\begin{lem}\label{app} Let $\bar{b}_\varepsilon=\bar{b}+\vec{\varepsilon}$, here $\vec{\varepsilon}=(\varepsilon, \varepsilon, \cdots, \varepsilon)\in\mathbb{R}^n$ and $\varepsilon>0$. Let $\bar{X}^\varepsilon(t)$ solve \eqref{neq1} with $\bar{X}^\varepsilon_0=\bar{X}_0$ and $\bar{b}_\varepsilon$ in place of $\bar{b}$. If  the conditions in Theorem \ref{main} hold,   then for any $T>0$, it holds that
$$\lim_{\varepsilon\to0^+}\mathbb{E}\sup_{t\in[0,T]}|\bar X^\varepsilon(t)-\bar X(t)|=0.$$
\end{lem}
The proof is standard, we omit it here.

\subsection{Necessary Conditions for Comparison Theorem }
In this subsection, we show the conditions in Theorem \ref{main}  are also necessary. To this end, we firstly introduce a lemma.
\begin{lem}\label{partial}
\begin{enumerate}
\item[(1)]For any $1\leq i\leq n$, $\xi,\eta\in\C$ with  $\xi^i(0)-D^i(\xi)=\eta^i(0)-D^i(\eta)$, there exists $\zeta\in\C$ such that $\zeta\leq \xi\land\eta$ and $\zeta^i(0)-D^i(\zeta)=\xi^i(0)-D^i(\xi)=\eta^i(0)-D^i(\eta)$.
\item[(2)] For $\mu,\nu\in \scr P_2(\C)$, there exists $\tilde{\mu}\in \scr P_2(\C)$ such that $\tilde{\mu}\leq_D \mu$ and $\tilde{\mu}\leq_D \nu$.
 \end{enumerate}
\end{lem}
\begin{proof} (1)
Fix $1\leq i\leq n$, $\xi,\eta\in\C$ with  $\xi^i(0)-D^i(\xi)=\eta^i(0)-D^i(\eta)$. Without loss of generality, assume $\xi^i(0)\leq \eta^i(0)$. If $D^i(\xi\wedge\eta)=D^i(\xi)$, let $\zeta=\xi\wedge\eta$. Otherwise, ${\bf (A1)}$ implies $D^i(\xi\wedge\eta)<D^i(\xi)$. Let $\mathbf{e}\in\C$ be defined by $\mathbf{e}_i(s)=1, s\in[-r_0,0], 1\leq i\leq n $. Define
\begin{align}\label{hhh}
h^i(r)=r-D^i(r\mathbf{e}), \ \  r\in\mathbb{R}.
\end{align} By {\bf(A5)} and $D(0)=0$, we have $h^i(r)\geq r(1-\kappa), r\geq 0$ and $h^i(r)\leq  r(1-\kappa), r\leq  0$. The continuity of $D$ implies that there exists a constant $v>0$ such that
\begin{align}\label{vvv}
h^i(v)=v-D^i(v\mathbf{e})=[(\xi\wedge\eta)^i(0)-D^i(\xi\wedge\eta)]-[\xi^i(0)-D^i(\xi)].
\end{align}
Let $\zeta=\xi\wedge\eta-v \mathbf{e}$, then it is clear that $\zeta\leq \xi\wedge\eta$. Moreover, it follows from \eqref{vvv}
\begin{align*}\zeta^i(0)-D^i(\zeta)&=(\xi\wedge\eta)^i(0)-v-(D^i(\xi\wedge\eta)- D^i(v\mathbf{e})) =\xi^i(0)-D^i(\xi).
\end{align*}

(2) Fix $\mu,\nu\in \scr P_2(\C)$. Let two $\C$-valued random variables  $(\Gamma_1,\Gamma_2)$ on $(\C^2,\B(\C^2),\mu\times \nu)$ be defined as $\Gamma_k(\xi_1,\xi_2)=\xi_k, k=1,2$. Then $\L_{\Gamma_1}|\mu\times\nu=\mu$ and  $\L_{\Gamma_2}|\mu\times\nu=\nu$.
 Basing on this, we can construct a $\C$-valued random variable $\tilde{\Gamma}$ on $(\C^2,\B(\C^2),\mu\times \nu)$ such that $\tilde{\Gamma}\leq \Gamma_k$ and $\tilde{\Gamma}(0)-D(\tilde{\Gamma})\leq \Gamma_k(0)-D(\Gamma_k)$, $k=1,2$. Let $\tilde{\mu}=\L_{\tilde{\Gamma}}|\mu\times\nu$. Then we have $\tilde{\mu}\leq_D \mu$ and $\tilde{\mu}\leq_D \nu$.

 In fact, for any $i=1,\cdots,n$, let $h^i$ be defined in \eqref{hhh}. For any $(\xi_1,\xi_2)\in\C^2$, let $$\alpha= -\left|\min_{s\in[-r_0,0]}(\xi_1\wedge\xi_2)(s)\right|,$$ then $\alpha \mathbf{e}\leq \xi_1\wedge\xi_2$.
Similarly, let $$\alpha_i= \frac{-|(\xi^i(0)-D^i(\xi))\wedge(\eta^i(0)-D^i(\eta))|}{1-\kappa},$$
then $$\alpha_i-D^i(\alpha_i \mathbf{e})\leq (\xi^i(0)-D^i(\xi))\wedge(\eta^i(0)-D^i(\eta)).$$
Let $\tilde{\alpha}=\alpha\wedge \min_i\alpha_i$ and $\tilde{\Gamma}(\xi_1,\xi_2)=\tilde{\alpha}\mathbf{e}, (\xi_1,\xi_2)\in\C^2$.
Thus, we finish the proof.
\end{proof}
\begin{thm}\label{main'} Let ({\bf{A1}})-{\bf(A5)} hold. Assume that
$\eqref{neq}$-$\eqref{neq1}$ is D-order-preserving
for any complete filtration probability space $(\OO,\{\F_t\}_{t\ge 0},\P)$ and $m$-dimensional Brownian motion $W(t)$ thereon.  Then  for any $1\leq i\leq n$, $\mu,\nu\in \scr P_2(\C)$ with $\mu\le_D \nu$, and $\xi,\eta\in\C$ with $\xi\le_D \eta$ and $\xi^i(0)-D^i(\xi)=\eta^i(0)-D^i(\eta)$, the following assertions hold:
\beg{enumerate}
\item[$(i')$]  $b_i(t,\xi,\mu)\leq \bar{b}_i(t, \eta, \nu)$ if  $b_i$ and $\bar b_i$ are continuous  at points $(t,\xi,\mu)$ and $(t, \eta,\nu)$ respectively.
\item[$(ii')$] For any $1\le j\le m$, $\sigma_{ij}(t,\xi,\mu)=\bar\si_{ij}(t, \eta,\nu)$ if $\si_{ij}$ and $\bar \si_{ij}$ are continuous   at points $(t,\xi,\mu)$ and $(t, \eta,\nu)$ respectively.
\end{enumerate}
Consequently,
when $b, \bar b$ and $\si, \bar \si$ are continuous on $[0,\infty)\times \C\times \scr P_2(\C)$, conditions $(i)$ with $\mu\leq_D\bar{\mu}$ in place of $\mu\leq \bar{\mu}$ and $(ii)$ hold.
\end{thm}

 We first observe that when $b, \bar b$ are continuous on $[0,\infty)\times \C\times\scr P_2(\C)$, $(i')$ implies $(i)$ with $\mu\leq_D\bar{\mu}$ in place of $\mu\leq \bar{\mu}$.  Next, we prove when $\si, \bar \si$ are continuous on $[0,\infty)\times \C\times\scr P_2(\C)$, $(ii')$ implies $(ii)$.

 Firstly, taking $\xi=\eta$ and $\mu=\nu$, by the continuity of $\si$ and $\bar \si$, $(ii')$ implies $\si=\bar\si$.

Let $\zeta$ and $\tilde{\mu}$ be in Lemma \ref{partial} associated to $\xi,\eta$ and $\mu,\nu$ and applying  $(ii')$ twice we obtain
$$\si_{ij}(t,\xi,\mu)=\si_{ij}(t,\zeta,\tilde{\mu})= \si_{ij}(t, \eta,\nu).$$
Since $\si=\bar\si$, this implies  $(ii)$.

\

 Now,  let $t_0\ge 0$,  $1\leq i\leq n$, $\mu,\nu\in \scr P_2(\C)$ with $\mu\le_D \nu$,   and $\xi,\eta\in\C$ with $\xi\le\eta$ and $\xi^i(0)-D^i(\xi)=\eta^i(0)-D^i(\eta)$.  To prove $(i')$ and $(ii')$ for $t=t_0$, we   construct a family of  complete filtration probability spaces $(\OO,\{\F_t^\vv\}_{t\geq 0},\P^\vv)_{\vv\in [0,1)}$, $m$-dimensional Brownian motion $W(t)$, and initial random variables $ X_{t_0}\le \bar X_{t_0}$ as follows.

Firstly, since $\mu\le_D \nu$, by Remark \ref{pc}, we may take $\pi_0\in {\mathbf C}(\mu,\nu)$ such that
\beq\label{OS1}\pi_0(\{(\xi_1,\xi_2)\in \C^2: \xi_1\le_D \xi_2\})=1.\end{equation} For any $\vv\in [0,1)$, let
\beq\label{OS2}\pi_\vv= (1-\vv)\pi_0+\vv \dd_{(\xi,\eta)},\end{equation}where $\dd_{(\xi,\eta)}$ is the Dirac measure at point $(\xi,\eta)$.
Let $\P_0$ be the standard Wiener measure on $\OO_0:=C([0,\infty)\to \mathbb{R}^m)$, and let $\F_{0,t}$ be the completion of $\si(\oo_0\mapsto \oo_0(s): s\le t)$ with respect to the Wiener measure. Then
the coordinate process $\{W_0(t)\}(\oo_0):= \oo_0(t),\ \oo_0\in\OO_0, t\ge 0$ is an $m$-dimensional Brwonian motion on the filtered probability space $(\OO_0, \{\F_t^0\}_{t\ge 0}, \P_0)$.

Next, for any $\vv\ge 0$, let
$\OO=\OO_0\times \C^2,\ \P^\vv= \P_0 \times \pi_\vv$ and $\F_t^\vv$ be the completion of $\F_{0,t}\times \B(\C^2)$ under the probability measure  $\P^\vv$. Then the process
$$\{W(t)\}(\oo):= \{W_0(t)\}(\oo_0)=\oo_0(t),\ \ t\ge 0, \oo=(\oo_0;\xi_1,\xi_2)\in \OO=\OO_0\times \C^2$$
is an $m$-dimensional Brownian motion on the complete   probability space $(\OO, \{\F_t^\vv\}_{t\ge 0}, \P^\vv)$.

Finally, let
$$X_{t_0}(\oo):= \xi_1,\ \ \bar X_{t_0}(\oo):= \xi_2,\ \ \oo=(\oo_0; \xi_1,\xi_2)\in \OO=\OO_0\times \C^2.$$
They are $\F_{t_0}^\vv$-measurable random variables with
\beq\label{XT}  \scr L_{X_{t_0}}|_{\P^\vv}=\mu_\vv:=\pi_\vv(\cdot\times\C),\ \ \  \scr L_{\bar X_{t_0}}|_{\P^\vv}=\nu_\vv:=\pi_\vv(\C\times\cdot).\end{equation}  By $\xi\le \eta$ with $\xi^i(0)-D^i(\xi)=\eta^i(0)-D^i(\eta)$ and \eqref{OS1}, \eqref{OS2}, we have
\beq\label{XT0} \P^\vv(X_{t_0}\le_D \bar X_{t_0})=\pi_\vv\big(\{(\xi_1,\xi_2)\in \C^2:\ \xi_1\le_D \xi_2\}\big)=1,\ \ \vv\in [0,1).\end{equation}
So, letting $(X_t,\bar X_t)_{t\ge t_0}$ be the segment process of the solution to \eqref{neq} and \eqref{neq1} with initial value $(X_{t_0}, \bar X_{t_0})$, the D-order preservation implies
\beq\label{OP}\P^\vv\big( X_t\le_D \bar X_t,\ t\ge t_0\big)=1,\ \ \vv\in [0,1).\end{equation}
Let $\E^\vv$ be the expectation for $\P^\vv$. With the above preparations, we are able to prove
  $(i')$ and $(ii')$ as follows.

\beg{proof}[Proof of $(i')$]  Let  $b_i, \bar b_i$ be continuous at points $(t_0, \xi,\mu)$ and $(t_0,\eta,\nu)$ respectively.
We intend to prove
 $b_i(t_0, \xi,\mu)\le b_i(t_0,\eta,\nu)$. Otherwise, there exists a constant $c_0>0$ such that
 \beq\label{CB} b_i(t_0, \xi,\mu)\ge c_0+ \bar b_i(t_0,\eta,\nu).\end{equation}
 Let $\mu_\vv,\nu_\vv$ be in \eqref{XT}.
 Obviously, $\{\mu_\vv,\nu_\vv\}_{\vv\in [0,1)}$ are bounded in $\scr P_2(\C)$ and,  as $\vv\to 0$,  $\mu_\vv\to\mu,\ \nu_\vv\to\nu$ weakly.
Consequently,
 $$\lim_{\vv\downarrow 0} \{\W_2(\mu_\vv,\mu)+\W_2(\nu_\vv,\nu)\}=0.$$
 Combining this with \eqref{CB} and the continuity of $b$ and $\bar{b}$, there exists $\vv\in (0,1)$ such that
 \beq\label{CB2} b_i(t_0, \xi,\mu_\vv)\ge \ff 1 2 c_0+  \bar b_i(t_0,\eta,\nu_\varepsilon)> \bar b_i(t_0,\eta,\nu_\varepsilon).\end{equation}

 Now, consider the event
\beq\label{AX} A:=\{X_{t_0}= \xi, \bar X_{t_0}=\eta\}\in \F_{t_0}^\vv.\end{equation}
 Then
 \beq\label{XT3} \P^\vv(A)\ge \vv \dd_{(\xi,\eta)}(\{(\xi,\eta)\})=\vv>0.\end{equation}
By \eqref{neq}, \eqref{neq1} and \eqref{OP}, for any $s\ge 0$, $\P^\vv$-a.s.
\beq\label{XA}\beg{split}
0&\ge X^i(t_0+s)-D^i(X_{t_0+s})-(\bar X^i(t_0+s)-D^i(\bar{X}_{t_0+s}))\\
&=X^i(t_0)-D^i(X_{t_0})-(\bar X^i(t_0)-D^i(\bar{X}_{t_0}))   \\
 &\quad +   \int_{t_0}^{t_0+s}b_i(r,X_r,\L_{X_r})\,\d r-\int_{t_0}^{t_0+s}\bar b_i(r,\bar X_r,\L_{\bar{X}_r})\,\d r  \\
&\quad +  \sum_{j=1}^m \int_{t_0}^{t_0+s}\sigma_{ij}(r,X_r,\L_{X_r})\,\d W_j(r)-\int_{t_0}^{t_0+s}\bar \sigma_{ij}(r,\bar X_r,\L_{\bar{X}_r})\,\d W_j(r).\end{split}\end{equation}
By {\bf(A2)} and the non-explosion of the solution  to \eqref{neq} and \eqref{neq1},  taking conditional expectation in \eqref{XA} with respect to $\F^\vv_{t_0}$,   we
obtain    $\P^\vv$-a.s.
\begin{equation*}\begin{split}
\E^\vv\bigg( \int_{t_0}^{t_0+s}b_i(r,X_r,\L_{X_r})\,\d r\bigg|\F^\vv_{t_0}\bigg) \leq \E^\vv\bigg(  \int_{t_0}^{t_0+s}\bar{b}_i(r,\bar{X}_r,\L_{\bar{X}_r})\,\d r\bigg|\F^\vv_{t_0}\bigg),\ \ s>0.
\end{split}\end{equation*} By \eqref{AX}, this implies
$$\E^\vv\bigg(\ff{1_A}s \int_{t_0}^{t_0+s}b_i(r,X_r,\L_{X_r})\,\d r\bigg|\F^\vv_{t_0}\bigg) \leq \E^\vv\bigg(  \ff{1_A}s\int_{t_0}^{t_0+s}\bar{b}_i(r,\bar{X}_r,\L_{\bar{X}_r})\,\d r\bigg|\F^\vv_{t_0}\bigg),\ \ s>0.$$
Combining this with the fact that  $b_i$ and $\bar b_i$ are continuous at points   $(t_0,\xi,\mu)$ and $(t_0, \eta,\nu)$ respectively,
 and using  {\bf(A2)}, the non-explosion and continuity of the solution  to \eqref{neq} and \eqref{neq1}, taking $s\downarrow 0$ we derive  $\P^\vv$-a.s.
 $$
\E^\vv\big(b_i(t_0,X_{t_0}, \scr L_{X_{t_0}}) \big|\F^\vv_{t_0}\big)\leq\E^\vv\big(\bar b_i(t_0,\bar X_{t_0}, \scr L_{\bar X_{t_0}}) \big|\F^\vv_{t_0}\big).$$
This together with \eqref{AX} and \eqref{XT} leads to  $\P^\vv$-a.s.
\begin{equation*}\begin{split}
b_i(t_0,\xi,\mu_\vv)1_{A}\leq \bar{b}_i(t_0,\eta,\nu_\vv)1_A,
\end{split}\end{equation*} which is impossible  according to \eqref{CB2} and \eqref{XT3}.
Therefore, $b_i(t_0, \xi,\mu)\le \bar b_i(t_0,\eta,\nu)$ has to be true.
\end{proof}

\beg{proof}[Proof of $(ii')$] Let $\si_{ij}$ and $\bar \si_{ij}$ be continuous at points $(t_0, \xi,\mu)$ and $(t_0,\eta,\nu)$ respectively.
If $\si_{ij}(t_0, \xi,\mu)\ne \bar\si_{ij}(t_0,\eta,\nu)$,   by {\bf (A3')}, there exist constants $c_1>0$ and $\vv\in (0,1)$ such that
\beq\label{SXT} |\si_{ij}(t_0, \xi, \mu_\vv)-\bar\si_{ij}(t_0, \eta,\nu_\vv)|^2\ge 2c_1>0.\end{equation}
For any $n,l\geq1$, let
\beg{align*}
&\tau=\inf\big\{t\ge t_0: |\si_{ij}(t, X_t, \scr L_{X_t}|_{\P^\vv})-\bar\si_{ij}(t, \bar X_t, \scr L_{\bar X_t}|_{\P^\vv})|^2\le c_1\big\},\\
&\tau_n=\inf\big\{t\ge t_0: X^i(t)-D^i(X_t)-(\bar X^i(t)-D^i(\bar{X}_t))\le -\frac{1}{n}\big\},\\
&\tau_{l,n}= \tau\land \tau_n\land\inf\big\{ t\ge t_0: |b_i(t, X_t, \scr L_{X_t}|_{\P^\vv})- \bar b_i(t, \bar X_t, \scr L_{\bar X_t}|_{\P^\vv})|\ge l\big\}.\end{align*} Let $g_n(s)=\e^{ns}-1.$ Then $g_n\in C_b^2((-\infty,0])$. By the D-order preservation we have $X^i_t\le_D \bar X^i_t,\ t\ge t_0$. So, Letting $Z^i(s)=(X^i-\bar X^i)(s)-D^i(X_{s})+D^i(\bar{X}_{s})$ and applying It\^o's formula, we obtain $\P^\vv$-a.s.
  \beg{align*}  &0\ge  \E^\vv  \big(g_n(Z^i( t\land {\tau}_{n,l}))|\F^\vv_{t_0}\big)\\
  & =
  g_n(Z^i( t_0))\\
  &+\E^\vv \bigg(\sum_{j=1}^m\int_{t_0}^{t\land{\tau}_{n,l}} g_n'(Z^i(s))\big(\sigma_{ij}(s,X_s,\L_{X_s}|_{\P^\vv})-\bar \sigma_{ij}(s,\bar X_s,\L_{\bar{X}_s}|_{\P^\vv})\big)\d W_j(s)\bigg|\F^\vv_{t_0}\bigg) \\
& + \E^\vv \bigg(\int_{t_{0}}^{t\land{\tau}_{n,l}} \bigg\{g_n'(Z^i(s)) \big(b_i(s,X_s,\L_{X_s}|_{\P^\vv})-\bar b_i(s,\bar X_s,\L_{\bar{X}_s}|_{\P^\vv})\big)\\
&\ +\ff {g_n''(Z^i(s))} 2\sum_{j=1}^m \big|\si_{ij}(s, X_s,\L_{X_s}|_{\P^\vv})-\bar\si_{ij}(s,\bar X_s,\L_{\bar{X}_s}|_{\P^\vv})\big|^2\bigg\}\d s\bigg|\F^\vv_{t_0}\bigg) \\
&\ge g_n(Z^i(t_0)) +\Big(\ff{n^2c_1}{2\e} -  nl\Big)\E^\vv (t\land{\tau}_{n,l}-t_0|\F^\vv_{t_0}),\ \ n, l\ge 1.\end{align*}
By \eqref{AX} and $\xi^i(0)-D^i(\xi)=\eta^i(0)-D^i(\eta),$ this implies
$$ 1_A \Big(\ff{n^2c_1}{2\e} -  nl\Big)\E^\vv (t\land{\tau}_{n,l}-t_0|\F^\vv_{t_0})\le -1_A g_n(Z^i(t_0) ) =-1_A g_n(0)=0$$ for all $n,l\ge 1$ and $t>t_0$.

Take $l\geq 2|b_i(t_0,\xi,\mu_\vv)-\bar{b}^i(t_0,\eta,\nu_\vv)|$ and $n>\frac{2\e l}{c_1}$, we obtain
\beq\label{LK}1_A \E^\vv (t\land{\tau}_{n,l}-t_0|\F_{t_0}^\vv)= 0,\ \ t>t_0.\end{equation}  But by ${\bf (A3)}$, \eqref{SXT}  and the continuity of the solution, on the set $A$ we have
$$\tau_{n,l}>t_0.$$ So, \eqref{LK} implies $\P^\vv(A)=0$, which contradicts \eqref{XT3}. Hence,  $$\si_{ij}(t_0, \xi,\mu)= \bar \si_{ij}(t_0, \eta,\nu).$$
\end{proof}

\beg{thebibliography}{99}

\bibitem{bai1} X. Bai, J. Jiang, \emph{Comparison theorems for neutral stochastic functional differential equations,} J. Differential Equ. 260(2016), 7250--7277.

\bibitem{BY} J. Bao, C. Yuan, \emph{Comparison theorem for stochastic differential delay equations with jumps,} Acta Appl. Math. 116(2011), 119--132.

\bibitem{CW} M.-F. Chen, F.-Y. Wang, \emph{On order-preservation and positive correlations for multidimensional diffusion processes,} Prob. Theory. Relat. Fields 95(1993), 421--428.

    \bibitem{GD} L. Gal'cuk, M. Davis, \emph{A note on a comparison theorem for equations with different diffusions,} Stochastics 6(1982), 147--149.

\bibitem{HRW} X. Huang, M. R\"{o}ckner, F.-Y. Wang, \emph{
Nonlinear Fokker--Planck equations for probability measures on path space and path-distribution dependent SDEs,}  arXiv:1709.00556.

    \bibitem{HW} X. Huang, F.-Y. Wang, \emph{Order-preservation for multidimensional stochastic functional differential equations with jumps,} J.  Evol. Equat. 14(2014),445--460.

\bibitem{IW} N. Ikeda, S. Watanabe, \emph{A comparison theorem for solutions of stochastic differential equations and its applications,} Osaka J. Math. 14(1977), 619--633.

\bibitem{KKO} T. Kamae, U.Krengel, G. L. O'Brien, \emph{Stochastic inequalities on partially ordered spaces,}
Ann. Probab. 5(1977), 899--912.

    \bibitem{M} X. Mao, \emph{A note on comparison theorems for stochastic differential equations with respect to semimartingales,} Stochastics 37(1991), 49--59.
\bibitem{O} G. L. O'Brien, \emph{A new comparison theorem for solution of stochastic differential equations,} Stochastics 3(1980), 245--249.

\bibitem{PY} S. Peng, Z. Yang, \emph{Anticipated backward stochastic differential equations,} Ann. Probab. 37(2009), 877--902.

\bibitem{PZ} S. Peng, X. Zhu, \emph{Necessary and sufficient condition for comparison theorem of 1-dimensional stochastic differential equations,} Stochastic Process. Appl. 116(2006), 370--380.

\bibitem{S} A.V. Skorohod, \emph{Studies in the theory of random process,} Addison-Wesley, 1965.


\bibitem{W} J.-M. Wang, \emph{Stochastic comparison for L\'evy-type processes,} J. Theor. Probab.  26(2013), 997--1019.

\bibitem{W1} F.-Y. Wang, \emph{The stochastic order and critical phenomena for superprocesses,} Infin. Dimens. Anal. Quantum Probab. Relat. Top. 9(2006), 107--128.

\bibitem{W16} F.-Y. Wang, \emph{Distribution-dependent SDEs for Landau type equations,} Stoch. Proc. Appl. 128(2018), 595-621.

\bibitem{Y} T. Yamada, \emph{On comparison theorem for solutions of stochastic differential equations and its applications,} J. Math. Kyoto Univ. 13(1973), 495--512.

    \bibitem{YMY}  Z. Yang, X. Mao, C. Yuan, \emph{Comparison theorem of one-dimensional stochastic hybrid systems,} Systems Control Lett. 57(2008), 56--63.

\bibitem{Z} X. Zhu, \emph{On the comparison theorem for multi-dimensional stochastic differential equations with jumps (in Chinese),} Sci. Sin. Math. 42(2012), 303--311.
\end{thebibliography}

\end{document}